\documentclass[11pt]{amsart}

\usepackage[T1]{fontenc}
\usepackage{amsmath,amssymb,amsfonts,amsthm, comment}
\usepackage{microtype}
\usepackage{flafter}
\usepackage{float}
\usepackage{tikz}
\usetikzlibrary{arrows.meta,calc,patterns,decorations.pathreplacing}
\usepackage[sort,compress]{cite}
\usepackage[hidelinks]{hyperref}

\numberwithin{equation}{section}
\newtheorem{theorem}[equation]{Theorem}
\newtheorem{lemma}[equation]{Lemma}
\newtheorem{corollary}[equation]{Corollary}
\newtheorem{proposition}[equation]{Proposition}
\newtheorem{remark}[equation]{Remark}
\newtheorem{example}[equation]{Example}

\newcommand{\N}{\mathbb{N}}
\newcommand{\CC}{\mathbb{C}}
\newcommand{\DD}{\mathbb{D}}

\title{Frames from Functional Calculus}

\author[I. Krishtal]{Ilya Krishtal}
\address{School of Mathematical and Statistical Sciences\\
Northern Illinois University\\
DeKalb, IL 60115, USA}
\email{ikrishtal@niu.edu}

\author[J. Mashreghi]{Javad Mashreghi}
\address{D\'epartement de math\'ematiques et de statistique\\
Universit\'e Laval\\
Qu\'ebec (Qu\'ebec) G1V 0A6, Canada}
\email{javad.mashreghi@mat.ulaval.ca}

\author[B. Miller]{Brendan Miller}
\address{Department of Mathematics\\
Vanderbilt University\\
1326 Stevenson Center, Nashville, TN 37240, USA}
\email{brendan.miller@vanderbilt.edu}

\date{}

\hypersetup{
  pdftitle={Frames from Functional Calculus},
  pdfauthor={Ilya Krishtal, Javad Mashreghi, Brendan Miller},
  pdfkeywords={Carleson frames, functional calculus, dynamical sampling,
  interpolating sequences, Stolz domains, frame perturbations}
}

\begin{document}

\begin{abstract}
We introduce frames generated by applying families of functions to a diagonal
operator through functional calculus, with diagonal entries forming a Carleson
interpolating sequence. First, for separated spectra contained in a Stolz
domain and an angular sector, we prove that systems of fractional operator
powers remain frames whenever the step size satisfies the natural angular
restriction. This extends earlier results for positive real spectra and permits
noninteger powers. 
We also
show that the angular restriction is sharp in general. Second, we give an
$\ell^2$-perturbation criterion for frames generated by more general continuous
functions, assuming a density condition on a curve containing the spectrum.
The results broaden the class of structured frames available in dynamical
sampling and related operator-orbit problems.
\end{abstract}


\maketitle

\section{Introduction}
Structured frames, such as Gabor, wavelet, or dynamical frames, have been a reliable cornerstone of frame theory research, as they frequently appear in applications and provide more sophisticated tools for analysis than general frames \cite{Christensen16}. In this paper, we introduce frames from functional calculus (FFC) which are a natural generalization of dynamical frames \cite{ACKM26, ACMT17}. As the most interesting case of (infinite) dynamical frames is the case of Carleson frames \cite{CHPS24, KM26}, we focus on FFC which generalize such frames.

A Carleson (interpolating) sequence is a sequence of complex numbers $\mathcal Z = \{z_j\}_{j \geq 0} \subset \DD$ that satisfies the \textit{Carleson condition}:
\begin{equation}\label{CCond}
\inf_{k} \prod_{j\neq k} \bigg|\frac{z_k - z_j}{1-\overline{z_j}z_k} \bigg| =: \delta > 0.
\end{equation}
The Carleson condition has been well studied, and the basic theory of interpolating sequences can be found, for example, in \cite{NN2012, Garnett81}. There are many partial characterizations of Carleson sequences, and in this paper we will restrict our attention to the subclass comprised of separated sequences lying in a Stolz domain. That is, we have
 \[
\inf_{j \neq k} \bigg| \frac{z_j - z_k}{1-\overline{z_j}z_k}\bigg| > 0,
\]
and $
\mathcal Z\subset \mathcal S_\alpha$ for some $\alpha \ge 1$ where
\[
\mathcal S_\alpha = \left\{z\in \DD: |1-z| \leq \alpha(1-|z|)\right\},\qquad \alpha\geq1.
\]
These two conditions taken together imply \eqref{CCond}, as shown in \cite[pp~157-158]{NN2012}.

\begin{figure}[t]
\centering
\begin{tikzpicture}[scale=2.35]
  \def\a{2.15}
  \fill[gray!8] (0,0) circle (1);
  \draw[gray!65] (0,0) circle (1);
  \draw[->] (-1.15,0) -- (1.18,0) node[right] {$\operatorname{Re} z$};
  \draw[->] (0,-1.08) -- (0,1.08) node[above] {$\operatorname{Im} z$};
  \path[fill=gray!18,draw=black!75,thick]
    plot[domain=-180:180,samples=180,variable=\t]
    ({((\a*\a-cos(\t))-sqrt((\a*\a-cos(\t))*(\a*\a-cos(\t))-(\a*\a-1)*(\a*\a-1)))/(\a*\a-1)*cos(\t)},
     {((\a*\a-cos(\t))-sqrt((\a*\a-cos(\t))*(\a*\a-cos(\t))-(\a*\a-1)*(\a*\a-1)))/(\a*\a-1)*sin(\t)}) -- cycle;
  \draw[black,very thick] (0,0) -- (1,0);
  \fill (1,0) circle (0.018) node[below right] {$1$};
  \fill (0,0) circle (0.012) node[below left] {$0$};
  \fill[black] (0.72,0.17) circle (0.018) node[above] {$z$};
  \draw[dashed,black!75] (1,0) -- (0.72,0.17);
  \draw[decorate,decoration={brace,amplitude=4pt},black!75] (1,0) -- (0.72,0.17)
    node[midway,above right=3pt] {$|1-z|$};
  \node[black] at (0.40,-0.36) {$\mathcal S_\alpha$};
  \node[gray!70] at (-0.63,0.62) {$\DD$};
\end{tikzpicture}
\caption{A schematic Stolz domain $\mathcal S_\alpha=\{z\in\DD: |1-z|\leq \alpha(1-|z|)\}$. The shaded set narrows non-tangentially toward the boundary point $1$}
\label{F:stolz-domain}
\end{figure}
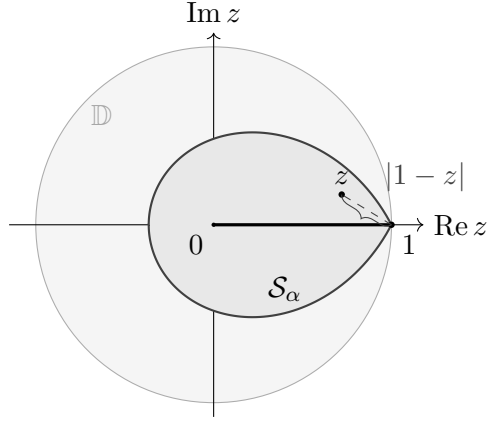

Given a Carleson sequence $\mathcal{Z}$, the \emph{Carleson operator} is the diagonal operator $D$ on $\ell^2(\N_0)$ formed with the sequence $\mathcal Z$ on its diagonal. Given the generating vector,
\[
g := \left( \sqrt{1-|z_j|^2} \right)_{j \geq 0} \in \ell^2(\N_0),
\]
the system $\{D^k g\}_{k = 0}^\infty$ forms a frame for $\ell^2(\N_0)$, typically called a \emph{Carleson frame}. It is known that if $A \in B(\ell^2)$ is normal and the system $\{A^k f\}_{k = 0}^\infty$ forms a frame for $\ell^2$, then there is a bounded invertible operator $P \in B(\ell^2)$ such that $PAP^{-1}$ is a Carleson operator and $P f$ is the corresponding generating vector \cite{ACMT17, CMPP20}.

Given a family $\mathcal F=(f_k)_{k\geq0}$ of continuous functions on a connected closed set containing $\mathcal Z$, our goal is to investigate conditions under which the sequence

\begin{equation}\label{E:seq-fk-1}
(f_k(D)g)_{j \geq 0} = \left( f_k(z_j)\sqrt{1-|z_j|^2} \right)_{j \geq 0}, \qquad k \geq 0,
\end{equation}
is a frame for $\ell^2(\N_0)$.

In \cite{CHPS24} it was shown that if $\mathcal{Z}$ is such that $\{|z_j|\}_{j \geq 0}$ is a strictly increasing Carleson sequence, then the system \eqref{E:seq-fk-1} forms a frame when $f_k(z) = z^{Nk}$ for every $N\in\N$. An immediate consequence of this is that if $z_0 \neq 0$, then one can also choose $f_k(z) = z^{Nk+j}$ for some values $N \in \N$, $j\in\{0,1,\ldots,N-1\}$. It was conjectured, and later proved in \cite{KM26}, that if $\mathcal{Z} \subset (0,1)$, then this result extends to the family $f_k(z) = z^{Nk + j_k}$ for $j_k \in [0,N)$.

Throughout the paper, non-integer powers are defined by the principal branch,
\[
z^\lambda=\exp(\lambda\operatorname{Log} z),\qquad \operatorname{Arg} z\in(-\pi,\pi),
\]
which is single-valued on each sector
    \[
    \DD_c=\{z=re^{i\theta}\in\DD\setminus\{0\}: \theta\in[-c,c]\},\qquad c\in[0,\pi).
    \]
Whenever the expression $\pi/c$ occurs, we use the convention $\pi/0=+\infty$.

One of the two main contributions of this paper extends the above-mentioned result from \cite{KM26} to encompass separated sequences in a Stolz domain.
We prove the following theorem in Section \ref{S:CSSD}.

\begin{theorem}\label{T:main1}
Assume that $\mathcal Z=\{z_j\}_{j\geq0}$ is a separated sequence contained
in $\mathcal S_\alpha\cap\DD_c$ for some $\alpha\geq1$ and $0\leq c<\pi$.
If $D$ is the Carleson operator associated with $\mathcal Z$ and $g$ is its
canonical generating vector, then
\[
\{D^{rk+j_k}g\}_{k\geq0}
\]
forms a frame for every $0<r<\pi/c$ and every choice $j_k\in[0,r)$.
\end{theorem}

Our second main contribution uses frame perturbation to obtain a wider class of FFC.  In Section~\ref{S:FFC}, we prove the following result.

\begin{theorem}\label{T:main2}
Let $\mathcal Z=\{z_j\}_{j\geq0}$ be a separated sequence contained in
$\mathcal S_\alpha\cap\DD_c$ for some $\alpha\geq1$ and $0\leq c<\pi$.
Let $D$ be the associated Carleson operator and let $g$ be its canonical
generating vector. Let $\Gamma=\gamma([0,1])\subset\overline{\DD}$, where
$\gamma:[0,1]\to\overline{\DD}$ is injective and continuous, and assume that
$\mathcal Z\subset\Gamma$. Let $\{f_k\}_{k\geq0}\subset C(\Gamma)$ satisfy
\begin{equation}\label{E:fk-unif-linfty4}
\overline{\operatorname{span}}\{f_k:k\geq0\}=C(\Gamma).
\end{equation}
Let $\Lambda=\{\lambda_k:k\geq0\}\subset[0,\infty)$ be countably infinite
and enumerated in strictly increasing order. Assume that there are constants
$r>0$ and $N\in\N$, with $r<\pi/c$, such that
\begin{equation}\label{E:counting_hyp}
1\leq\bigl|\Lambda\cap[rm,r(m+1))\bigr|\leq N,
\qquad m\in\N_0.
\end{equation}
Assume further that there exists $M\geq0$ such that
\begin{equation}\label{E:fk-ell2-perturbation-main}
\sum_{k=0}^{\infty}\left|f_k(z_j)-z_j^{\lambda_k}\right|^2\leq M^2,
\qquad j\geq0.
\end{equation}
Then $\{f_k(D)g\}_{k\geq0}$ forms a frame for $\ell^2(\N_0)$.
\end{theorem}

Below we will use the pointwise square-summability of $\{f_k(z)\}_k$, which follows from the 
hypotheses of the above theorem. Indeed, for each $z\in\DD$, the block-counting condition \eqref{E:counting_hyp} gives
\[
\sum_{k=0}^{\infty}|z|^{2\lambda_k}
\leq N\sum_{m=0}^{\infty}|z|^{2rm}
=\frac{N}{1-|z|^{2r}}.
\]
Consequently, \eqref{E:fk-ell2-perturbation-main} implies
\begin{equation}\label{E:fk-unif-linfty3}
\sum_{k=0}^{\infty}|f_k(z_j)|^2
\leq 2M^2+\frac{2N}{1-|z_j|^{2r}}<\infty,
\qquad j\geq0.
\end{equation}

\section{Carleson Sequences in a Stolz Domain}
\label{S:CSSD}

The goal of this section is to prove Theorem \ref{T:main1}. We begin, however, with the necessary properties of the Stolz domains and Carleson sequences in them.

In the literature, a more general Stolz condition is often considered. It is typically assumed that there is some finite set $A\subset \partial \DD$ such that
    \[
    \sup_{j\geq 0} \frac{\min\limits_{a \in A}\{|z_j-a|\}}{1-|z_j|}  < \infty.
    \]
Indeed, it is known that any separated sequence $\mathcal Z\subset \DD$ satisfying this general Stolz condition is a Carleson sequence. In this paper, however, we only consider Stolz domains accumulating on $1$, i.e.~$A=\{1\}$. This allows us to conclude that if $\mathcal Z \subset \mathcal S_\alpha$ for some $\alpha\geq 1$ then
$\lim\limits_{j\to \infty} z_j = 1$. It also allows us to establish the following crucial properties of $\mathcal Z \subset \mathcal S_\alpha$.

\begin{lemma}\label{L:Stolz}
    Let $\mathcal{S}_\alpha$ be a Stolz domain and define
    \[
    \DD_c=\{z=re^{i\theta}\in\DD\setminus\{0\}: \theta\in[-c,c]\},\qquad c\in[0,\pi).
    \]
   If $\mathcal{Z}\subset \mathcal{S}_\alpha\cap\DD_c$, then for every $R>0$ there exists a constant $C_R>0$ such that
   \[
   |1-z^s|\leq C_R(1-|z|^R),\qquad z\in\mathcal Z,
   \]
   for all $s\in[0,R]$.  In particular, for each fixed $r>0$ there is $\widetilde\alpha>0$ such that $\{z^r:z\in\mathcal Z\}\subset\mathcal S_{\widetilde\alpha}$.
\end{lemma}

\begin{proof}
Fix $R>0$. Choose $0<\varepsilon<1$ so small that the closed disc
$\overline{B_\varepsilon(1)}$ is disjoint from the origin and from the branch
cut $(-\infty,0]$. If $z\in\mathcal Z\cap B_\varepsilon(1)$, then the segment
$\Gamma_z$ joining $1$ to $z$ is contained in $B_\varepsilon(1)$. Hence
\[
\sup_{0\leq s\leq R}\sup_{w\in\Gamma_z}
\left|\frac{d}{dw}w^s\right|<\infty
\]
uniformly in $z$. It follows that, for some constant $C_1>0$,
\[
|1-z^s|\leq C_1|1-z|
\leq C_1\alpha(1-|z|),
\qquad 0\leq s\leq R.
\]
For $0\leq t<1$,
\[
1-t\leq \max\{1,R^{-1}\}(1-t^R),
\]
so the desired estimate holds on $\mathcal Z\cap B_\varepsilon(1)$.

Now suppose that $z\in\mathcal Z\setminus B_\varepsilon(1)$. The Stolz
condition gives
\[
1-|z|\geq\frac{|1-z|}{\alpha}\geq\frac{\varepsilon}{\alpha},
\]
and therefore $|z|\leq q:=1-\varepsilon/\alpha<1$ after decreasing
$\varepsilon$, if necessary. Thus
\[
|1-z^s|\leq2
\leq\frac{2}{1-q^R}(1-|z|^R),
\qquad 0\leq s\leq R.
\]
Combining the two regions proves the first assertion. Taking $s=R=r$ gives
\[
|1-z^r|\leq C_r(1-|z|^r),
\]
which is the Stolz condition for the powered set, with a possibly larger
constant.
\end{proof}

 The outline of the proof of Theorem \ref{T:main1} is essentially the same as in \cite{KM26} for the case of $\mathcal Z \subset (0,1)$. First, we show that if $\{D^kg\}_{k\geq0}$ is a Carleson frame whose spectrum is separated and lies in a Stolz domain, then $\{D^{rk}g\}_{k\geq0}$ is a frame for $0<r<\pi/c$. We then show that, for every choice $j_k\in[0,r)$, the sequence $\{D^{rk+j_k}g\}_{k\geq0}$ is an outer frame on a sufficiently far spectral tail. Finally, we use this result to show that such outer frames are, in fact, frames for $\ell^2(\N_0)$.

\begin{proposition}\label{P:carleson_powers}
    Let $\mathcal{S}_\alpha$ be a Stolz domain and $c \in [0,\pi)$. Let $\{z_j\}_{j\geq 0} = \mathcal{Z} \subset \mathcal{S}_\alpha \cap \DD_c$ be a separated sequence.
    Then the sequence $\{z_j^r\}_{j \geq 0}$ is a Carleson sequence for every $0 < r < \pi/c$. Consequently, if $D$ is the Carleson operator corresponding to $\mathcal{Z}$, then $\{D^{rk}g\}_{k=0}^\infty$ is a Carleson frame for $0 < r < \pi/c$.
\end{proposition}

    \begin{proof}
Fix \(0<r<\pi/c\), and write
\[
\rho_{\DD}(z,w)
    :=
    \left|
        \frac{z-w}{1-\overline zw}
    \right|
\]
for the pseudo-hyperbolic distance. Since \(\mathcal Z\) is
separated, there is a constant
\[
\delta_{\mathcal Z}
    :=
    \inf_{i\neq j}\rho_{\DD}(z_i,z_j)>0.
\]

By Lemma~\ref{L:Stolz}, there exists
\(\widetilde{\alpha}>0\) such that
\[
\mathcal Z^{(r)}
    :=
    \{z_j^r:j\geq0\}
    \subset \mathcal S_{\widetilde{\alpha}}.
\]
It therefore suffices to prove that \(\mathcal Z^{(r)}\) is
separated.

Consider the principal-branch power map
\[
F_r(z)=z^r.
\]
The function \(F_r\) is holomorphic in a neighborhood of \(1\), and
\[
F_r'(1)=r\neq0.
\]
Thus, by the holomorphic inverse function theorem, there are
neighborhoods \(U\) and \(V\) of \(1\) such that
\[
F_r:U\longrightarrow V
\]
is biholomorphic. Let \(G_r:V\to U\) denote its inverse. Choose a
disc \(V_0\) centered at \(1\) with
\(\overline{V_0}\subset V\), and put
\[
U_0=G_r(V_0).
\]
After shrinking \(V_0\), if necessary, we may assume that \(V_0\)
is convex and that \(U_0\) lies in a sufficiently small neighborhood
of \(1\).

Since \(G_r'\) is bounded on \(\overline{V_0}\), there is a constant
\(C_r>0\) such that, for all \(z,w\in U_0\),
\[
\begin{aligned}
|z-w|
    &=
    \big|G_r(F_r(z))-G_r(F_r(w))\big|  \\
    &\leq
    C_r\big|F_r(z)-F_r(w)\big|
    =
    C_r|z^r-w^r|.
\end{aligned}
\]
Consequently,
\begin{equation}
\label{E:power-numerator-comparison}
|z^r-w^r|
    \geq C_r^{-1}|z-w|,
    \qquad z,w\in U_0.
\end{equation}

Shrinking \(U_0\) once more, we may also assume that, whenever
\(z,w\in U_0\), the point
\[
u=\overline zw
\]
belongs to a convex neighborhood \(W_0\) of \(1\) on which the
principal logarithm is analytic and satisfies
\[
\operatorname{Log}(\overline zw)
    =
    \overline{\operatorname{Log}z}
    +
    \operatorname{Log}w.
\]
It follows that
\[
\overline{z^r}\,w^r
    =
    (\overline zw)^r.
\]
Since the derivative of \(u\mapsto u^r\) is bounded on
\(\overline{W_0}\), there is a constant \(K_r>0\) such that
\begin{equation}
\label{E:power-denominator-comparison}
\big|1-\overline{z^r}\,w^r\big|
    =
    \big|1-(\overline zw)^r\big|
    \leq
    K_r|1-\overline zw|,
    \qquad z,w\in U_0.
\end{equation}

Because \(z_j\to1\), there is \(J\in\N\) such that \(z_j\in U_0\)
for every \(j\geq J\). Combining
\eqref{E:power-numerator-comparison} and
\eqref{E:power-denominator-comparison}, we obtain, for distinct
\(i,j\geq J\),
\[
\begin{aligned}
\rho_{\DD}(z_i^r,z_j^r)
    &=
    \frac{|z_i^r-z_j^r|}
         {|1-\overline{z_i^r}z_j^r|} \\
    &\geq
    \frac{1}{C_rK_r}
    \frac{|z_i-z_j|}
         {|1-\overline{z_i}z_j|} \\
    &\geq
    \frac{\delta_{\mathcal Z}}{C_rK_r}>0.
\end{aligned}
\]
Thus the powered tail \(\{z_j^r\}_{j\geq J}\) is separated.

We next verify that the powered points are distinct. Write
\[
z_j=\varrho_j e^{i\theta_j},
\qquad
\theta_j\in[-c,c].
\]
If \(z_i^r=z_j^r\), then \(\varrho_i=\varrho_j\) and
\[
r(\theta_i-\theta_j)\in2\pi\mathbb Z.
\]
However,
\[
|r(\theta_i-\theta_j)|
    \leq 2rc<2\pi,
\]
so \(r(\theta_i-\theta_j)=0\). Hence \(\theta_i=\theta_j\), and
therefore \(z_i=z_j\). Since the original sequence has distinct
points, so does the powered sequence.

Finally, for every fixed \(i<J\),
\[
\rho_{\DD}(z_i^r,z_j^r)
    \longrightarrow
    \frac{|z_i^r-1|}
         {|1-\overline{z_i^r}|}
    =1
    \qquad (j\to\infty).
\]
Together with the separation of the tail and the fact that there
are only finitely many indices below \(J\), this shows that the full
sequence \(\{z_j^r\}_{j\geq0}\) is separated. Since it also lies in
the Stolz domain \(\mathcal S_{\widetilde{\alpha}}\), it is a
Carleson sequence.

Let
\[
g_r
    :=
    \left(
        \sqrt{1-|z_j|^{2r}}
    \right)_{j\geq0}.
\]
Because \(\{z_j^r\}_{j\geq0}\) is a Carleson sequence,
\[
\{D^{rk}g_r\}_{k\geq0}
    =
    \{(D^r)^kg_r\}_{k\geq0}
\]
is its canonical Carleson frame.

Define the diagonal operator \(P\) by
\[
P\delta_j
    =
    \sqrt{
        \frac{1-|z_j|^2}
             {1-|z_j|^{2r}}
    }\,
    \delta_j.
\]
The scalar function
\[
p_r(t)
    =
    \sqrt{
        \frac{1-t^2}
             {1-t^{2r}}
    },
    \qquad 0\leq t<1,
\]
extends continuously to \([0,1]\) by setting
\[
p_r(1)=\frac{1}{\sqrt r}.
\]
It is strictly positive on \([0,1]\), and is therefore bounded above and
bounded away from zero there. Consequently, \(P\) is bounded and
invertible. Moreover, \(P\) commutes with \(D\) and
\[
Pg_r=g.
\]
It follows that
\[
D^{rk}g
    =
    D^{rk}Pg_r
    =
    PD^{rk}g_r,
    \qquad k\geq0.
\]
Hence \(\{D^{rk}g\}_{k\geq0}\) is the image of the frame
\(\{D^{rk}g_r\}_{k\geq0}\) under the bounded invertible operator
\(P\), and is therefore itself a frame.
\end{proof}

\begin{remark}
The sector hypothesis is used to prevent collisions under the map $z\mapsto
z^r$. More generally, the same proof works whenever the finitely many powered
points outside the tail remain distinct. The restriction $r<\pi/c$ is natural
and, as Remark~\ref{R:sharpness} shows, sharp in the stated uniform form.
\end{remark}

For the next step, we use the following perturbation result, which is an analog of \cite[Lemma~2.1]{KM26}.

\begin{lemma}\label{L:truncation}
Let $\{z_j\}_{j\geq0}$ be a Carleson sequence, and let $D$ and $g$ be the
associated Carleson operator and canonical generating vector. Assume that
$\{h_k(D)g\}_{k\geq0}$ is a frame with lower frame bound $A>0$. Suppose that
there are $M\geq0$ and $J'\in\N_0$ such that
\begin{equation}\label{E:ell2-perturb-truncation}
\sum_{k=0}^{\infty}|f_k(z_j)-h_k(z_j)|^2\leq M^2,
\qquad j\geq J'.
\end{equation}
Then there is an integer $J\geq J'$ such that
$\{P_{V_J}f_k(D)g\}_{k\geq0}$ is a frame for
\[
V_J=\overline{\operatorname{span}}\{\delta_j:j\geq J\}.
\]
Equivalently, $\{f_k(D)g\}_{k\geq0}$ is an outer frame for $V_J$.
\end{lemma}

\begin{proof}
For any $J\geq J'$, we estimate the perturbation required by the standard frame perturbation lemma as follows:
\begin{align*}
\sum_{k=0}^{\infty}\left\|P_{V_J}\big(f_k(D)g-h_k(D)g\big)\right\|^2
&=\sum_{k=0}^{\infty}\sum_{j=J}^{\infty}(1-|z_j|^2)\left|f_k(z_j)-h_k(z_j)\right|^2\\
&=\sum_{j=J}^{\infty}(1-|z_j|^2)\sum_{k=0}^{\infty}\left|f_k(z_j)-h_k(z_j)\right|^2\\
&\leq M^2\sum_{j=J}^{\infty}(1-|z_j|^2).
\end{align*}
Since every Carleson sequence satisfies the Blaschke condition, the last tail can be made smaller than $A$ by choosing $J$ sufficiently large.  The sequence $\big\{P_{V_J}h_k(D)g\big\}_{k\geq0}$ has lower frame bound at least $A$ on $V_J$, and \cite[Lemma~1.2]{KM26} therefore implies that $\big\{P_{V_J}f_k(D)g\big\}_{k\geq0}$ is a frame for $V_J$.
\end{proof}

We note that the $\ell^2$ perturbation condition \eqref{E:ell2-perturb-truncation} in Lemma \ref{L:truncation} is not automatic for arbitrary complex Carleson spectra. For instance, if $z_j=(-1)^j(1-2^{-(j+1)})$, then
\[
\sum_{k=0}^{\infty}\left|z_j^{2k}-z_j^{2k+1}\right|^2
=\frac{|1-z_j|^2}{1-|z_j|^4},
\]
which is unbounded along the odd indices.  The situation improves when the spectrum lies in a Stolz domain.

\begin{lemma}\label{L:Stolz_outer_frames}
Assume that $\mathcal Z$ is a separated sequence contained in
$\mathcal S_\alpha\cap\DD_c$ for some $\alpha\geq1$ and $c\in[0,\pi)$.
If $0<r<\pi/c$ and $j_k\in[0,r)$ for every $k\geq0$, then
$\{D^{rk+j_k}g\}_{k\geq0}$ is an outer frame for
$V_J=\overline{\operatorname{span}}\{\delta_j:j\geq J\}$ for some integer
$J$.
\end{lemma}

\begin{proof}

    By Proposition~\ref{P:carleson_powers}, the sequence
$\{D^{rk}g\}_{k\geq0}$ is a frame for every $0<r<\pi/c$.
By Lemma~\ref{L:Stolz},

\[
|1-z_j^{j_k}|\leq C_r(1-|z_j|^r),\qquad j,k\geq0.
\]
Consequently,
\begin{align*}
\sum_{k=0}^{\infty}\left|z_j^{rk+j_k}-z_j^{rk}\right|^2
&\leq C_r^2(1-|z_j|^r)^2\sum_{k=0}^{\infty}|z_j|^{2rk}\\
&=C_r^2\frac{(1-|z_j|^r)^2}{1-|z_j|^{2r}}
\leq C_r^2.
\end{align*}
Lemma \ref{L:truncation}, with $h_k(z)=z^{rk}$ and $f_k(z)=z^{rk+j_k}$, gives the desired outer frame.
\end{proof}

 Given a Carleson sequence $\mathcal Z$ and the corresponding operator $D$ and vector $g$, we shall denote by $\Phi_{\mathcal{Z}}^\Lambda$ the synthesis operator associated with the sequence $\{D^{rk+j_k} g\}_{k \geq 0}$.
   Observe that it has a matrix representation whose $(j,k)$-entry is
\begin{equation}
       \label{E:synth_matr}
        z^{rk+j_k}_j\sqrt{1-|z_j|^2}.
\end{equation}
  The following proposition, which serves as an extension of \cite[Proposition~2.7]{KM26} to the current setting, is the last key piece that we require for the proof of Theorem \ref{T:main1}.

\begin{proposition}\label{P:fractional-replacement}
Given $\alpha\geq1$, $c\in[0,\pi)$, and $0<r<\pi/c$, let
$\Lambda=\{rk+j_k:k\geq0\}$, where $j_k\in[0,r)$, and let
$\mathcal Z_J=\{z_j\}_{j\geq J}\subset\DD_c$. If
$\Phi^\Lambda_{\mathcal Z_J}$ is bounded and surjective, then one can choose
$z'_0,\ldots,z'_{J-1}\in\mathcal S_\alpha\cap\DD_c$, avoiding at most
countably many points at each step, such that
$\Phi^\Lambda_{\mathcal Z'}$ is surjective, where
\[
\mathcal Z'=\{z'_0,\ldots,z'_{J-1}\}\cup\mathcal Z_J.
\]
\end{proposition}

\begin{proof}
It is enough to show that one row can be added while preserving surjectivity.
Suppose that $\Phi^\Lambda_{\mathcal W}$ is bounded and surjective, where
$\mathcal W$ consists of the infinite tail together with finitely many
replacement points already chosen.

We first show that the column norms tend to zero. The exponents $\lambda_k=rk+j_k$
are strictly increasing and tend to infinity. Fix $k_0$. For $k\geq k_0$ and
$w\in\mathcal W$,
\[
(1-|w|^2)|w|^{2\lambda_k}
\leq(1-|w|^2)|w|^{2\lambda_{k_0}}.
\]
The sequence on the right is summable because it is the square of the
$k_0$-th column of the bounded operator $\Phi^\Lambda_{\mathcal W}$.
Dominated convergence therefore yields
\[
\|\Phi^\Lambda_{\mathcal W}e_k\|\longrightarrow0.
\]
If $\Phi^\Lambda_{\mathcal W}$ were injective, then surjectivity and the
bounded inverse theorem would make it bounded below, contradicting the last
limit. Hence we may choose
\[
0\neq a=(a_k)_{k\geq0}\in\ker\Phi^\Lambda_{\mathcal W}.
\]

On the slit disc $\Omega=\DD\setminus(-1,0]$, define
\[
\vartheta_a(z)=\sum_{k=0}^{\infty}a_kz^{\lambda_k}
\]
using the principal branch. If $K\subset\Omega$ is compact and
$|z|\leq q<1$ on $K$, then Cauchy--Schwarz gives
\[
\sum_{k=0}^{\infty}|a_k||z|^{\lambda_k}
\leq\|a\|_2\left(\sum_{k=0}^{\infty}q^{2rk}\right)^{1/2}.
\]
Thus the series converges uniformly on compact subsets of $\Omega$, so
$\vartheta_a$ is analytic there. It is not identically zero. Indeed, let $m$
be the first index for which $a_m\neq0$. If $\vartheta_a(e^{-x})=0$ for all
sufficiently large $x$, then
\[
0=e^{\lambda_mx}\vartheta_a(e^{-x})
=a_m+\sum_{k>m}a_ke^{-(\lambda_k-\lambda_m)x}.
\]
The final sum tends to zero by Cauchy--Schwarz and the linear growth of
$\lambda_k$, which would force $a_m=0$. Hence the zeros of $\vartheta_a$ are
discrete in $\Omega$ and therefore countable.

Choose the next point $z'\in\mathcal S_\alpha\cap\DD_c$ outside this zero set
and outside the previously chosen points. The associated new row is a bounded
functional on $\ell^2$, since
\[
\sum_{k=0}^{\infty}|z'|^{2\lambda_k}
\leq\sum_{k=0}^{\infty}|z'|^{2rk}<\infty.
\]
Let $R_{z'}$ denote this row, including the factor $\sqrt{1-|z'|^2}$. Then
$R_{z'}a\neq0$. Given an old target $y$ and a desired new coordinate $t$,
choose $u$ with $\Phi^\Lambda_{\mathcal W}u=y$ and set
\[
v=u+\frac{t-R_{z'}u}{R_{z'}a}\,a.
\]
Since $\Phi^\Lambda_{\mathcal W}v=y$
and $R_{z'}v=t$, the augmented operator $\Phi^\Lambda_{\mathcal W \cup\{z'\}}$ is surjective. Iterating this argument finitely
many times proves the proposition.
\end{proof}

We are now ready to proceed with the final step of the outline and prove Theorem \ref{T:main1}.

\begin{proof}[Proof of Theorem~\ref{T:main1}]
By Proposition~\ref{P:carleson_powers}, $\{D^{rk}g\}_{k\geq0}$ is a frame.
The estimate in Lemma~\ref{L:Stolz_outer_frames} also gives
\[
\begin{aligned}
\sum_{k=0}^{\infty}\|(D^{rk+j_k}-D^{rk})g\|^2
&=\sum_{j=0}^{\infty}(1-|z_j|^2)
  \sum_{k=0}^{\infty}|z_j^{rk+j_k}-z_j^{rk}|^2\\
&\leq C_r^2\sum_{j=0}^{\infty}(1-|z_j|^2)<\infty.
\end{aligned}
\]
Hence $\{(D^{rk+j_k}-D^{rk})g\}_{k\geq0}$ is Bessel, and so is
$\{D^{rk+j_k}g\}_{k\geq0}$.

Lemma~\ref{L:Stolz_outer_frames} implies that the tail synthesis operator
$\Phi^\Lambda_{\mathcal Z_J}$ is surjective for some $J$. By
Proposition~\ref{P:fractional-replacement}, choose
$z'_0,\ldots,z'_{J-1}\in\mathcal S_\alpha\cap\DD_c$ such that
$\Phi^\Lambda_{\mathcal Z'}$ is surjective, where
\[
\mathcal Z'=\{z'_0,\ldots,z'_{J-1}\}\cup\mathcal Z_J.
\]
The difference $\Phi^\Lambda_{\mathcal Z}-\Phi^\Lambda_{\mathcal Z'}$ has
range in $\operatorname{span}\{\delta_0,\ldots,\delta_{J-1}\}$ and is
therefore compact. By \cite[Theorem~22.2.1]{Christensen16},
$\{D^{rk+j_k}g\}_{k\geq0}$ is a frame sequence.

It remains to prove completeness. 
Since
\[
    \lambda_k=rk+j_k=rk+O(1),
\]
comparison with the harmonic series in every multiplicative interval
\([t,\mu t]\) gives
    $L(\Lambda)=1/r$,
where 
\[
L(\Lambda):= \inf_{\mu > 1} \frac{1}{\log \mu} \limsup_{t \to \infty} \sum_{\Lambda \cap [t,\mu t]}\frac{1}{\lambda}
\]
is 
the logarithmic block density of the set $\Lambda=\{rk+j_k:k\geq0\}$.
Thus $r<\pi/c$ gives $L(\Lambda)=1/r>c/\pi$, and completeness follows from
\cite[Theorem~3.5]{KM26}. The frame sequence is therefore a frame for
$\ell^2(\N_0)$.
\end{proof}

\begin{remark}
The final step above differs from the positive-real argument in \cite{KM26}.
Rather than using the M\"untz--Sz\'asz theorem to add deleted rows, it uses
analyticity on the slit disc, discreteness of zero sets, and a compact
perturbation result for frame sequences.
\end{remark}

\begin{corollary}\label{C:Stolz_general}
Assume that $\mathcal Z=\{z_j\}_{j\geq0}$ is a separated sequence contained
in $\mathcal S_\alpha\cap\DD_c$ for some $\alpha\geq1$ and $c\in[0,\pi)$.
Let $D$ and $g$ be the associated Carleson operator and canonical generating vector. Let $\Lambda\subset[0,\infty)$ be countably infinite. If there are $r>0$ and
$N\in\N$, with $r<\pi/c$, such that
\[
1\leq\bigl|\Lambda\cap[rm,r(m+1))\bigr|\leq N,
\qquad m\in\N_0,
\]
then $\{D^\lambda g\}_{\lambda\in\Lambda}$ is a frame for
$\ell^2(\N_0)$.
\end{corollary}

\begin{proof}
Set $I_m=[rm,r(m+1))$ and
$n_m=|\Lambda\cap I_m|$. For each $m$, enumerate the points of
$\Lambda\cap I_m$ as
\[
\Lambda\cap I_m=\{\lambda_{m,1},\ldots,\lambda_{m,n_m}\}.
\]
For $1\leq\ell\leq N$, define the $\ell$-th partial track by
\[
\Lambda_\ell
=\{\lambda_{m,\ell}:m\in\N_0,\ n_m\geq\ell\}.
\]
Then the tracks are pairwise disjoint and
$\Lambda=\bigcup_{\ell=1}^N\Lambda_\ell$.

Because $n_m\geq1$, the first track has exactly one point in every block.
Writing $\lambda_{m,1}=rm+j_m$, with $j_m\in[0,r)$, and applying
Theorem~\ref{T:main1}, we see that
$\{D^\lambda g\}_{\lambda\in\Lambda_1}$ is a frame. Let $A_1>0$ and
$B_1<\infty$ be lower and upper frame bounds for this track.

For $2\leq\ell\leq N$, complete $\Lambda_\ell$ by choosing an arbitrary
point $\widetilde\lambda_{m,\ell}\in I_m$ whenever $n_m<\ell$, and setting
$\widetilde\lambda_{m,\ell}=\lambda_{m,\ell}$ otherwise. The completed set
\[
\widetilde\Lambda_\ell
=\{\widetilde\lambda_{m,\ell}:m\in\N_0\}
\]
has exactly one point in every block, so Theorem~\ref{T:main1} implies that
$\{D^\lambda g\}_{\lambda\in\widetilde\Lambda_\ell}$ is a frame. Its
subfamily indexed by $\Lambda_\ell$ is therefore Bessel; let $B_\ell$ be a
Bessel bound.

For every $x\in\ell^2(\N_0)$,
\[
\sum_{\lambda\in\Lambda}|\langle x,D^\lambda g\rangle|^2
\leq\left(\sum_{\ell=1}^N B_\ell\right)\|x\|^2,
\]
while the first track gives
\[
\sum_{\lambda\in\Lambda}|\langle x,D^\lambda g\rangle|^2
\geq A_1\|x\|^2.
\]
Thus the full family is a frame.
\end{proof}

    \begin{remark}
\label{R:sharpness}
Assume that \(0<c<\pi\). The angular restriction
\[
    r<\frac{\pi}{c}
\]
in Theorem~\ref{T:main1} is sharp in the following uniform sense.

Fix \(R\geq \pi/c\) and put
\[
    \vartheta=\frac{\pi}{R}\in(0,c].
\]
Choose \(0<\rho<1\), and define
\[
    \zeta_{+}=\rho e^{i\vartheta},
    \qquad
    \zeta_{-}=\rho e^{-i\vartheta},
    \qquad
    x_n=1-2^{-(n+2)},\quad n\geq0.
\]
After enumerating the points so that \(z_0=\zeta_{+}\) and
\(z_1=\zeta_{-}\), let
\[
    \mathcal Z=\{\zeta_{+},\zeta_{-}\}
      \cup\{x_n:n\geq0\}.
\]
The real sequence \(\{x_n\}_{n\geq0}\) is separated. In fact, if
\(n<m\), then
\[
\left|
    \frac{x_n-x_m}{1-x_nx_m}
\right|
\geq \frac{1}{3}.
\]
Moreover,
\[
    \rho_{\DD}(\zeta_{\pm},x_n)
    =
    \left|
        \frac{\zeta_{\pm}-x_n}
             {1-\overline{\zeta_{\pm}}x_n}
    \right|
    \longrightarrow 1
    \qquad (n\to\infty),
\]
so adjoining the two points \(\zeta_{+}\) and \(\zeta_{-}\)
preserves separation. Also \(\mathcal Z\subset\DD_c\), and
\(\mathcal Z\subset\mathcal S_\alpha\) for
\[
    \alpha=
    \max\left\{
        1,
        \frac{|1-\zeta_{+}|}{1-\rho},
        \frac{|1-\zeta_{-}|}{1-\rho}
    \right\}.
\]
Thus \(\mathcal Z\) satisfies the geometric hypotheses of Theorem~\ref{T:main1}.

On the other hand, using the principal branch,
\[
    \zeta_{+}^{\,R}
    =
    \rho^R e^{i\pi}
    =
    -\rho^R
    =
    \rho^R e^{-i\pi}
    =
    \zeta_{-}^{\,R}.
\]
Let \(D\) be the diagonal operator associated with \(\mathcal Z\), and let
\(g\) be its canonical generating vector. Choosing \(j_k=0\) for
every \(k\geq0\), we obtain
\[
\begin{aligned}
    (D^{Rk}g)_0
    &=
    \zeta_{+}^{\,Rk}\sqrt{1-\rho^2}  \\
    &=
    \zeta_{-}^{\,Rk}\sqrt{1-\rho^2}
    =
    (D^{Rk}g)_1,
    \qquad k\geq0.
\end{aligned}
\]
Consequently,
\[
    \big\langle \delta_0-\delta_1,D^{Rk}g\big\rangle=0,
    \qquad k\geq0.
\]
Hence \(\{D^{Rk}g\}_{k\geq0}\) is not complete and therefore cannot
be a frame. In particular, the endpoint \(R=\pi/c\) cannot generally
be included, and the conclusion also fails in general for every
\(R>\pi/c\).
\end{remark}

\section{Frames from Functional Calculus Perturbations}
\label{S:FFC}

In this section, we turn our attention to more general FFC and prove Theorem \ref{T:main2}. The outline of the proof is the same as for Theorem \ref{T:main1}. Thus, we begin with an analog of Lemma \ref{L:Stolz_outer_frames} for this setting.

\begin{lemma}\label{C:truncation-Stolz}
Assume that $\mathcal Z=\{z_j\}_{j\geq0}$ is a separated sequence contained
in $\mathcal S_\alpha\cap\DD_c$ for some $\alpha\geq1$ and $c\in[0,\pi)$.
Let $D$ and $g$ be the associated Carleson operator and canonical generating vector.
Let $\Lambda=\{\lambda_k:k\geq0\}\subset[0,\infty)$ be countably infinite
and enumerated increasingly. Suppose that there are $r>0$ and $N\in\N$, with
$r<\pi/c$, such that
\[
1\leq\bigl|\Lambda\cap[rm,r(m+1))\bigr|\leq N,
\qquad m\in\N_0.
\]
If, for some $M\geq0$,
\begin{equation}\label{E:fk-ell2-perturbation-tail}
\sum_{k=0}^{\infty}\left|f_k(z_j)-z_j^{\lambda_k}\right|^2\leq M^2,
\qquad j\geq0,
\end{equation}
then $\{f_k(D)g\}_{k\geq0}$ is an outer frame for
$V_J=\overline{\operatorname{span}}\{\delta_j:j\geq J\}$ for some integer
$J$.
\end{lemma}

\begin{proof}
By Corollary \ref{C:Stolz_general}, the sequence $\{D^{\lambda_k}g\}_{k\geq0}$ forms a frame. Apply Lemma \ref{L:truncation} with $h_k(z)=z^{\lambda_k}$ and use \eqref{E:fk-ell2-perturbation-tail}.
\end{proof}

The final step of the proof of Theorem \ref{T:main2}, however, cannot be replicated from the previous section since we do not have an analog of \cite[Theorem~3.5]{KM26} for this setting. Fortunately, we are still able to follow the blueprint of \cite[Section 2]{KM26}.

\begin{remark}
One might be tempted to think that in lieu of the last step of the proof, it would suffice to notice that $(f_k(D)g)_{k \geq 0}$ is an outer frame for both $V_J$ and its orthogonal complement. However, this fact itself does not, in general, guarantee that the system is a frame. Indeed, the system $\{(1,1), (2,2)\}$ is an outer frame for both $\{0\}\times\CC$ and $\CC\times \{0\}$, but clearly it is not a frame for $\CC^2$.
\end{remark}

\begin{proof}[Proof of Theorem~\ref{T:main2}]
Let $J$ be the integer guaranteed by Lemma \ref{C:truncation-Stolz}. Thus $\{f_k(D)g\}_{k\geq0}$ is an outer frame for $V_J$ with frame bounds $B\geq A>0$. Put
\[
\Phi_J=\left(f_k(z_j)\sqrt{1-|z_j|^2}\right)_{j\geq J,\,k\geq0},
\]
and
\[
\Phi=\Phi_0=\left(f_k(z_j)\sqrt{1-|z_j|^2}\right)_{j\geq0,\,k\geq0}.
\]
The operator $\Phi_J$ is bounded and surjective. By \eqref{E:fk-unif-linfty3}, $\Phi$ is obtained from $\Phi_J$ by adding finitely many square-summable rows. Hence $\Phi$ is bounded, and it only remains to prove surjectivity.

For $c=(c_k)_{k\geq0}\in\ell^2$, define
\[
\theta_c(z_j)=\sum_{k=0}^{\infty}c_k f_k(z_j),\qquad j\geq0,
\]
which is well defined by \eqref{E:fk-unif-linfty3}. Then
\[
\Phi_n c=\sum_{j=n}^{\infty}\theta_c(z_j)\sqrt{1-|z_j|^2}\,\delta_j,
\qquad n\in\N_0.
\]
Since $\Phi_J$ is a frame synthesis operator with bounds $B\geq A>0$,
\begin{equation}\label{E:norm-phi-inverse}
\|\Phi_J^*(\Phi_J\Phi_J^*)^{-1}\|\leq\frac{\sqrt B}{A}.
\end{equation}
Set
\[
\eta=\left((1-|z_{J-1}|^2)\sum_{k=0}^{\infty}|f_k(z_{J-1})|^2\right)^{1/2},
\qquad
M'=\left(\sum_{j=0}^{\infty}(1-|z_j|^2)\right)^{1/2},
\]
which is finite by the Blaschke condition. Choose
\[
0<\varepsilon<\frac{1}{2\left(1+\eta M'\frac{\sqrt B}{A}\right)}.
\]

Let $K_J=\overline{\{z_j:j\geq J\}}\subset\gamma([0,1])$. The point $z_{J-1}$ is separated from $K_J$: the only possible accumulation point of the tail is $1$, while $z_{J-1}\neq1$. Hence there is a continuous function $\varphi$ on $\gamma([0,1])$ such that
\[
\varphi(z_{J-1})=(1-|z_{J-1}|^2)^{-1/2},
\qquad
\varphi|_{K_J}=0.
\]
By the density assumption \eqref{E:fk-unif-linfty4}, choose a finitely supported sequence $b=(b_k)_{k\geq0}$ so that $\theta_b=\sum b_k f_k$ approximates $\varphi$ closely enough to ensure
\begin{equation}\label{E:theta-atj-1}
\left|\sqrt{1-|z_{J-1}|^2}\,\theta_b(z_{J-1})-1\right|<\varepsilon
\end{equation}
and
\begin{equation}\label{E:theta-atj-big}
|\theta_b(z_j)|<\varepsilon,
\qquad j\geq J.
\end{equation}
Put
\[
c=\Phi_J^*(\Phi_J\Phi_J^*)^{-1}\Phi_J b.
\]
Then
\begin{equation}\label{E:phib=phic}
\Phi_J c=\Phi_J b=\sum_{j=J}^{\infty}\theta_b(z_j)\sqrt{1-|z_j|^2}\,\delta_j,
\end{equation}
and, by \eqref{E:norm-phi-inverse} and \eqref{E:theta-atj-big},
\begin{align}
\|c\|
&\leq\frac{\sqrt B}{A}\|\Phi_J b\|\notag\\
&=\frac{\sqrt B}{A}\left(\sum_{j=J}^{\infty}|\theta_b(z_j)|^2(1-|z_j|^2)\right)^{1/2}\notag\\
&\leq \varepsilon\frac{\sqrt B}{A}M'.\label{E:estimation-c}
\end{align}
According to \eqref{E:phib=phic},
\[
\Phi_{J-1}(b-c)=\sqrt{1-|z_{J-1}|^2}\big(\theta_b(z_{J-1})-\theta_c(z_{J-1})\big)\delta_{J-1}.
\]
Moreover,
\begin{align*}
\|\Phi_{J-1}(b-c)\|
&=\left|\sqrt{1-|z_{J-1}|^2}\big(\theta_b(z_{J-1})-\theta_c(z_{J-1})\big)\right|\\
&\geq1-\varepsilon-
\sqrt{1-|z_{J-1}|^2}\,|\theta_c(z_{J-1})|\\
&\geq1-\varepsilon-\|c\|\eta\\
&\geq1-\varepsilon\left(1+\eta M'\frac{\sqrt B}{A}\right)>0.
\end{align*}
Thus $\delta_{J-1}$ belongs to the range of $\Phi_{J-1}$. Since $\Phi_J$ is surjective, this implies that $\Phi_{J-1}$ is surjective. Repeating the argument finitely many times adds the rows $J-2,J-3,\ldots,0$, so $\Phi$ is surjective. Therefore $\{f_k(D)g\}_{k\geq0}$ is a frame.
\end{proof}

We conclude the paper with examples illustrating Theorem \ref{T:main2}.

\begin{example}
\label{E:illustrative-main2}
\begin{enumerate}
    \item
\emph{A real triangular perturbation.}

Let
\[
z_j=1-2^{-(j+1)},\qquad j\geq0.
\]
Then $\mathcal Z=\{z_j\}_{j\geq0}\subset(0,1)$ is a separated Stolz sequence.  Let $D$ be the associated Carleson operator and let $g=\left(\sqrt{1-z_j^2}\right)_{j\geq0}$.  Take
\[
\lambda_k=k,
\qquad
\Lambda=\N_0,
\]
so the counting hypothesis \eqref{E:counting_hyp} holds with $r=1$ and $N=1$.  If $a=(a_k)_{k\geq1}\in\ell^2$, define
\[
f_0(z)=1,
\qquad
f_k(z)=z^k+a_kz^{k-1},\quad k\geq1.
\]
Then, the identity $z^k=f_k-a_kz^{k-1}$ shows inductively that every monomial lies in $\operatorname{span}\{f_k: k \in \mathbb{N}_0\}$; hence \eqref{E:fk-unif-linfty4} follows from the Weierstrass theorem. Finally,
\[
\sum_{k=0}^{\infty}|f_k(z)-z^k|^2
=\sum_{k=1}^{\infty}|a_k|^2z^{2k-2}
\leq\|a\|_2^2.
\]
Thus Theorem \ref{T:main2} implies that $\{f_k(D)g\}_{k\geq0}$ is a frame for $\ell^2(\N_0)$.

\medskip
\item
\emph{A non-real Stolz spectrum.}

Fix, for example, $q=1/2$ and $\tau=1/2$, and set
\[
s_j=q^{j+2},
\qquad
z_j=1-s_j+i\tau s_j,
\qquad j\geq0.
\]
Then $|1-z_j|\asymp s_j$ and $1-|z_j|\asymp s_j$, so $\mathcal Z\subset\mathcal S_\alpha$ for some $\alpha$.  Since the $s_j$ decrease geometrically, the pseudo-hyperbolic distances between distinct $z_j$ are bounded below; hence $\mathcal Z$ is separated.  The arguments of the points are bounded by some $c<\pi$, so $\mathcal Z\subset\DD_c$.

\begin{figure}[H]
\centering
\begin{tikzpicture}[scale=2.45]
  \def\a{2.15}
  \fill[gray!8] (0,0) circle (1);
  \draw[gray!65] (0,0) circle (1);
  \path[fill=gray!16,draw=black!65,thick]
    plot[domain=-180:180,samples=180,variable=\t]
    ({((\a*\a-cos(\t))-sqrt((\a*\a-cos(\t))*(\a*\a-cos(\t))-(\a*\a-1)*(\a*\a-1)))/(\a*\a-1)*cos(\t)},
     {((\a*\a-cos(\t))-sqrt((\a*\a-cos(\t))*(\a*\a-cos(\t))-(\a*\a-1)*(\a*\a-1)))/(\a*\a-1)*sin(\t)}) -- cycle;
  \draw[->] (-1.08,0) -- (1.22,0) node[right] {$\operatorname{Re} z$};
  \draw[->] (0,-1.08) -- (0,1.15) node[above] {$\operatorname{Im} z$};
  \draw[black,thick] (0.75,0.125) -- (1,0);
  \foreach \x/\y in {0.750/0.125,0.875/0.0625,0.938/0.03125,0.969/0.015625,0.984/0.0078125,0.992/0.00390625}{
    \fill[black] (\x,\y) circle (0.018);
  }
  \fill (1,0) circle (0.015) node[below right] {$1$};
  \node[black] at (0.77,0.23) {$z_j=1-s_j+i\tau s_j$};
  \node[gray!70] at (0.33,-0.16) {$\mathcal S_\alpha$};
\end{tikzpicture}
\caption{The non-real example: a geometrically spaced spectrum approaching $1$ along a non-tangential line segment inside a Stolz domain}
\label{F:nonreal-stolz-spectrum}
\end{figure}
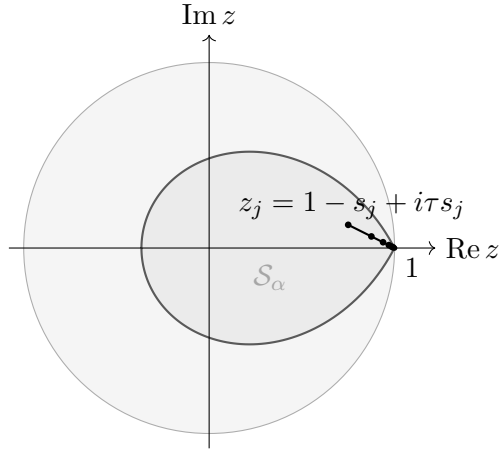

Let $\gamma(t)=1-s_0t+i\tau s_0t$, $0\leq t\leq1$, so that $\mathcal Z\subset\gamma([0,1])$.  Use the same functions
\[
f_0(z)=1,
\qquad
f_k(z)=z^k+a_kz^{k-1},\quad a\in\ell^2,
\]
and  $\Lambda=\N_0$.  Polynomials are dense on the line segment $\gamma([0,1])$ because it is an affine image of an interval, and the same triangular argument gives density of $\operatorname{span}\{f_k:k\geq0\}$.  Moreover,
\[
\sum_{k=0}^{\infty}|f_k(z_j)-z_j^k|^2
\leq\|a\|_2^2,
\qquad j\geq0.
\]
Thus Theorem \ref{T:main2} applies and yields a frame $\{f_k(D)g\}_{k\geq0}$ for a non-real Stolz spectrum.

\medskip

\item

\emph{A curved non-real Stolz spectrum.}

The previous example lies on a line segment.  To get a genuinely non-linear geometry, fix $q=1/2$, set $s_j=q^{j+2}$, and put
\[
z_j=1-s_j+4is_j^2,\qquad j\geq0.
\]
Then $\mathcal Z=\{z_j\}_{j\geq0}$ lies on the parabolic arc
\[
\Gamma=\{1-t+4it^2:0\leq t\leq1/4\},
\]
which is not contained in any affine line.
Moreover,
\[
|1-z_j|=s_j\sqrt{1+16s_j^2}\asymp s_j,
\qquad
1-|z_j|\asymp s_j,
\]
so $\mathcal Z\subset\mathcal S_\alpha$ for a suitable $\alpha\geq1$ (in fact, for every $\alpha >1$, all but finitely many points of $\mathcal Z$ belong to $\mathcal S_\alpha$).  The arguments of the points are bounded by some $c<\pi$, hence $\mathcal Z\subset\DD_c$.  Finally, the geometric spacing of the parameters implies separation: if $j<k$, then $s_k\leq qs_j$, the map $t\mapsto1-t+4it^2$ is bi-Lipschitz on $[0,1/4]$, and
\[
|z_j-z_k|\asymp |s_j-s_k|,
\qquad
|1-\overline{z_k}z_j|\asymp s_j+s_k.
\]
Therefore the pseudo-hyperbolic distances between distinct points are bounded below by a positive constant.

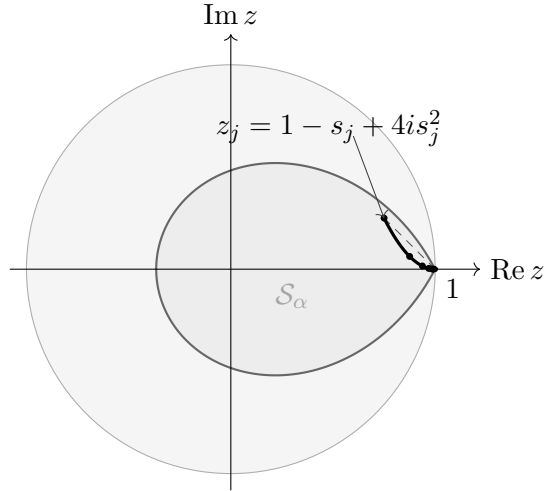
\begin{figure}[H]
\centering
\begin{tikzpicture}[scale=2.7, line cap=round, line join=round]
  \def\a{2.15}
  \fill[gray!8] (0,0) circle (1);
  \draw[gray!65] (0,0) circle (1);
  \path[fill=gray!14,draw=black!60,thick]
    plot[domain=-180:180,samples=180,variable=\t]
    ({((\a*\a-cos(\t))-sqrt((\a*\a-cos(\t))*(\a*\a-cos(\t))-(\a*\a-1)*(\a*\a-1)))/(\a*\a-1)*cos(\t)},
     {((\a*\a-cos(\t))-sqrt((\a*\a-cos(\t))*(\a*\a-cos(\t))-(\a*\a-1)*(\a*\a-1)))/(\a*\a-1)*sin(\t)}) -- cycle;
  \draw[->] (-1.08,0) -- (1.22,0) node[right] {$\operatorname{Re}z$};
  \draw[->] (0,-1.08) -- (0,1.150) node[above] {$\operatorname{Im}z$};
  \draw[black!65,dashed] (0.75,0.25) -- (1,0);
  \draw[black,very thick,domain=0:1,samples=80,variable=\u]
    plot ({1-0.25*\u},{0.25*\u*\u});
  \foreach \u in {1,0.5,0.25,0.125,0.0625,0.03125}{
    \fill[black] ({1-0.25*\u},{0.25*\u*\u}) circle (0.017);
  }
  \fill (1,0) circle (0.014) node[below right] {$1$};
  \node[black,align=center] at (0.48,0.69) {$z_j=1-s_j+4is_j^2$};
  \draw[->,black!75] (0.60,0.65) -- (0.75,0.25);
  \node[gray!70] at (0.30,-0.13) {$\mathcal S_\alpha$};
\end{tikzpicture}
\caption{A non-real Carleson spectrum on a parabolic arc.  The dashed chord emphasizes that the spectrum is not contained in a line segment}
\label{F:curved-nonreal-spectrum}
\end{figure}

Let
\[
\gamma(u)=1-\frac{u}{4}+\frac{i}{4}u^2,\qquad 0\leq u\leq1,
\]
so that $\gamma(q^j)=z_j$.  Take again $\Lambda=\N_0$ and
\[
f_0(z)=1,\qquad f_k(z)=z^k+a_kz^{k-1},\quad k\geq1,
\]
where $a=(a_k)_{k\geq1}\in\ell^2$.  The $\ell^2$ perturbation estimate is the same as in the preceding examples.  The density condition follows from Mergelyan's theorem \cite{Mergelyan52}, because $\gamma([0,1])$ is a Jordan arc and hence has connected complement; equivalently, polynomials are dense in $C(\gamma([0,1]))$.  The triangular relation again transfers this density from monomials to the family $\{f_k\}_{k\geq0}$.  Consequently Theorem~\ref{T:main2} applies, and
$\{f_k(D)g\}_{k\geq0}$ is a frame for $\ell^2(\N_0)$ with a spectrum that
is not contained in any line segment.

\medskip
\item
\emph{A non-polynomial functional-calculus family.}

Retain the parabolic spectrum and arc $\Gamma$ from the preceding example.
Define
\[
h(z)=e^{1-z},
\qquad
f_k(z)=h(z)z^k=e^{1-z}z^k,
\qquad k\geq0,
\]
and take $\lambda_k=k$. Since $h$ has no zeros on $\Gamma$, multiplication
by $h$ is a bounded invertible operator on $C(\Gamma)$. Mergelyan's theorem, the connectedness of $\CC\setminus\Gamma$, and the
empty interior of $\Gamma$ give
\[
\overline{\operatorname{span}}\{1,z,z^2,\ldots\}=C(\Gamma),
\]
and therefore
\[
\overline{\operatorname{span}}\{f_k:k\geq0\}=C(\Gamma).
\]

It remains to verify the perturbation hypothesis. Put
\[
R=\max_{z\in\Gamma}|1-z|.
\]
Using $|e^w-1|\leq e^{|w|}|w|$ and the Stolz
condition, we obtain
\[
\begin{aligned}
\sum_{k=0}^{\infty}|f_k(z_j)-z_j^k|^2
&=\frac{|e^{1-z_j}-1|^2}{1-|z_j|^2}\\
&\leq e^{2R}\frac{|1-z_j|^2}
{(1-|z_j|)(1+|z_j|)}\\
&\leq e^{2R}\alpha^2
\frac{1-|z_j|}{1+|z_j|}
\leq e^{2R}\alpha^2.
\end{aligned}
\]
Theorem~\ref{T:main2} now shows that
\[
\{e^{I-D}D^kg\}_{k\geq0}
\]
is a frame. This example involves a genuinely non-polynomial functional
calculus on a non-real curved Carleson spectrum.

\end{enumerate}
\end{example}

The preceding examples admit direct proofs of the frame property using triangular
coefficient transformations or a common invertible multiplier.
The next construction has neither of these forms; moreover, for
large values of its parameter $\kappa$ the standard global perturbation criterion fails. The example presents a more vivid illustration of the tail-restoration mechanism of Theorem \ref{T:main2}.

\begin{example}[A localized nonanalytic perturbation]
\label{E:localized-perturbation}

We begin with the same parabolic Jordan arc as before:
\[
    \Gamma=\gamma([0,1]),
    \qquad
    \gamma(t)
    =
    1-\frac{t}{4}+\frac{i t^2}{4},
    \qquad 0\leq t\leq1.
\]
For \(j\geq0\), set
\[
    t_j=2^{-j},
    \qquad
    z_j=\gamma(t_j)
    =
    1-\frac{t_j}{4}+\frac{i t_j^2}{4}.
\]
Thus
\[
    z_j
    =
    1-2^{-(j+2)}+4i\,2^{-2(j+2)},
\]
and this is precisely the curved non-real spectrum considered in the
preceding example. In particular, \(\mathcal Z=\{z_j\}_{j\geq0}\) is a
separated sequence contained in
\(\mathcal S_\alpha\cap D_c\) for suitable
\(\alpha\geq1\) and \(0\leq c<\pi\).

We first construct continuous functions localized at the individual
spectral points. Put
\[
    a_j=\frac{t_j+t_{j+1}}{2},
    \qquad j\geq0,
\]
and
\[
    b_0=1,
    \qquad
    b_j=\frac{t_{j-1}+t_j}{2},
    \qquad j\geq1.
\]
For every \(j\geq0\), choose a continuous hat function
\(\varphi_j\in C([0,1])\) such that
\[
    0\leq\varphi_j\leq1,
    \qquad
    \varphi_j(t_j)=1,
    \qquad
    \varphi_j(t)=0
    \quad\text{for }t\notin[a_j,b_j].
\]
The functions may be chosen piecewise linear and zero at every
endpoint other than \(t_0=b_0\). Their support sets are pairwise
disjoint. Define
\[
    \psi_j(\gamma(t))=\varphi_j(t),
    \qquad 0\leq t\leq1.
\]
Since \(\gamma\) is injective, each \(\psi_j\) is well defined and
belongs to \(C(\Gamma)\). Moreover,
\[
    \|\psi_j\|_\infty=1,
    \qquad
    \psi_k(z_j)=\delta_{kj},
    \qquad j,k\geq0.
\]

Let
\[
    \varepsilon_k=2^{-(k+1)^2},
    \qquad k\geq0,
\]
fix a constant \(\kappa>0\), and define
\[
    f_k(z)
    =
    z^k+\kappa\varepsilon_k\psi_k(z),
    \qquad z\in\Gamma,\quad k\geq0.
\]
We take
\[
    \lambda_k=k,
    \qquad
    \Lambda=\mathbb N_0,
\]
so that the block-counting condition \eqref{E:counting_hyp} holds with \(r=1\) and \(N=1\).

At the spectral points, the functions satisfy the particularly simple
identity
\begin{equation}
\label{E:localized-values}
    f_k(z_j)
    =
    z_j^k+\kappa\varepsilon_k\delta_{kj}.
\end{equation}
Consequently,
\begin{equation}
\label{E:localized-vectors}
    f_k(D)g
    =
    D^kg
    +
    \kappa\varepsilon_k
    \sqrt{1-|z_k|^2}\,\delta_k.
\end{equation}

We now verify the remaining hypotheses of Theorem~\ref{T:main2}.

The pointwise perturbation condition \eqref{E:fk-ell2-perturbation-main} is immediate:
\begin{equation}
\label{E:localized-l2-perturbation}
\begin{aligned}
    \sum_{k=0}^{\infty}
        |f_k(z_j)-z_j^k|^2
    &=
    \kappa^2\varepsilon_j^2                         \\
    &=
    \kappa^2 4^{-(j+1)^2}
    \leq \kappa^2,
    \qquad j\geq0.
\end{aligned}
\end{equation}

It remains to prove the density condition
$\overline{\operatorname{span}}
    \{f_k:k\geq0\}
    =
    C(\Gamma)$.
Suppose that a finite complex Borel measure \(\mu\) on \(\Gamma\)
annihilates every \(f_k\). Set
\[
    m_k=\int_\Gamma z^k\,d\mu(z),
    \qquad k\geq0.
\]
Since
\[
    0
    =
    \int_\Gamma f_k\,d\mu
    =
    m_k
    +
    \kappa\varepsilon_k
    \int_\Gamma\psi_k\,d\mu,
\]
we obtain
\[
    |m_k|
    \leq
    \kappa\|\mu\|\varepsilon_k
    =
    \kappa\|\mu\|2^{-(k+1)^2}.
\]
In particular,
\begin{equation}
\label{E:superexponential-moments}
    \limsup_{k\to\infty}|m_k|^{1/k}=0.
\end{equation}

Consider the Cauchy transform
\[
    \mathcal C_\mu(w)
    =
    \int_\Gamma\frac{d\mu(z)}{w-z},
    \qquad w\in\mathbb C\setminus\Gamma.
\]
For \(|w|>1\), expansion of the Cauchy kernel gives
\[
    \mathcal C_\mu(w)
    =
    \sum_{k=0}^{\infty}\frac{m_k}{w^{k+1}}.
\]
By \eqref{E:superexponential-moments}, the function
\[
    H(w)
    :=
    \sum_{k=0}^{\infty}\frac{m_k}{w^{k+1}}
\]
is holomorphic on \(\mathbb C\setminus\{0\}\). Since
\(0\notin\Gamma\) and a Jordan arc does not separate the plane,
    $\mathbb C\setminus(\Gamma\cup\{0\})$
is connected. The identity theorem therefore gives
\[
    \mathcal C_\mu(w)=H(w),
    \qquad
    w\in\mathbb C\setminus(\Gamma\cup\{0\}).
\]

Choose \(\rho>0\) so small that  $\overline{B_\rho(0)}\cap\Gamma=\varnothing$.
The Cauchy transform \(\mathcal C_\mu\) is holomorphic on
\(B_\rho(0)\), whereas
\[
    H(w)=\sum_{k=0}^{\infty}\frac{m_k}{w^{k+1}}
\]
agrees with \(\mathcal C_\mu\) on the punctured disc
\(B_\rho(0)\setminus\{0\}\). Hence the singularity of \(H\) at
\(0\) is removable. Since the displayed series is the Laurent
expansion of \(H\) at \(0\), all its negative Laurent coefficients
must vanish. Therefore
\[
    m_k=0,\qquad k\geq0.
\]
Thus \(\mu\) annihilates every polynomial. Since \(\Gamma\) is a
Jordan arc, Mergelyan's theorem implies that the polynomials are
dense in \(C(\Gamma)\). Consequently, \(\mu=0\).

The Hahn--Banach theorem and the Riesz
representation theorem now imply that
$\overline{\operatorname{span}}
    \{f_k:k\geq0\}
    =
    C(\Gamma)$.

All the assumptions of Theorem~\ref{T:main2} are therefore satisfied,
and we conclude that
    $\{f_k(D)g\}_{k\geq0}$
is a frame for \(\ell^2(\mathbb N_0)\) for every \(\kappa>0\).
\end{example}

\section*{Acknowledgments}  The first author was supported in part by the Fulbright Global Scholar Award; he thanks the second author for his wonderful hospitality in Qu\'ebec. The second author acknowledges support from the Canada Research Chairs Program CRC-2022-00097 and
the NSERC Discovery Grant RGPIN-2024-04232.
We also wish to thank Ole Christensen and Marzieh Hasannasab for pointing out the compact-perturbation argument to us and for further useful discussions. Finally, the authors acknowledge the use of generative AI for assistance in reviewing and improving this manuscript. 

\section*{Statements and Declarations}
\noindent\textbf{Conflict of interest statement.}
The authors declare that there are no conflicts of interest.

\medskip
\noindent\textbf{Data availability.}
No datasets were generated or analyzed during the current study.


\begin{thebibliography}{99}

\bibitem{ACKM26}
Aldroubi, A., Cabrelli, C., Krishtal, I., Molter, U.:
Dynamical sampling: a survey.
\emph{La Matematica} \textbf{5}, Article 37 (2026).
\url{https://doi.org/10.1007/s44007-026-00215-y}

\bibitem{ACMT17}
Aldroubi, A., Cabrelli, C., Molter, U., Tang, S.:
Dynamical sampling.
\emph{Appl. Comput. Harmon. Anal.} \textbf{42}, 378--401 (2017).
\url{https://doi.org/10.1016/j.acha.2015.08.014}

\bibitem{CMPP20}
Cabrelli, C., Molter, U., Paternostro, V., Philipp, F.:
Dynamical sampling on finite index sets.
\emph{J. Anal. Math.} \textbf{140}, 637--667 (2020).
\url{https://doi.org/10.1007/s11854-020-0099-2}

\bibitem{Christensen16}
Christensen, O.:
\emph{An Introduction to Frames and Riesz Bases}, 2nd edn.
Birkh\"auser, Cham (2016).
\url{https://doi.org/10.1007/978-3-319-25613-9}

\bibitem{CHPS24}
Christensen, O., Hasannasab, M., Philipp, F.M., Stoeva, D.:
The mystery of Carleson frames.
\emph{Appl. Comput. Harmon. Anal.} \textbf{72}, Article 101659 (2024).
\url{https://doi.org/10.1016/j.acha.2024.101659}

\bibitem{Garnett81}
Garnett, J.B.:
\emph{Bounded Analytic Functions}.
Pure and Applied Mathematics, vol.~96. Academic Press, New York (1981)


\bibitem{KM26}
Krishtal, I., Miller, B.:
Demystifying Carleson frames.
\emph{Appl. Comput. Harmon. Anal.} \textbf{80}, Article 101811 (2026).
\url{https://doi.org/10.1016/j.acha.2025.101811}

\bibitem{Mergelyan52}
Mergelyan, S.N.:
Uniform approximations to functions of a complex variable.
\emph{Uspekhi Mat. Nauk} \textbf{7}, 31--122 (1952)

\bibitem{NN2012}
Nikol'skii, N.K.:
\emph{Treatise on the Shift Operator: Spectral Function Theory}.
Springer, Berlin (1986).
\url{https://doi.org/10.1007/978-3-642-70151-1}

\end{thebibliography}
\end{document}